\documentclass[11 pt,amstex,3p,times]{article}
\usepackage{hyperref}
\usepackage{enumerate,fullpage}
\usepackage{amssymb,amsmath,graphicx,amsthm}
\usepackage{color}

\newcommand{\be}{\begin{eqnarray*}}
	\newcommand{\en}{\end{eqnarray*}}
\newcommand{\bes}{\begin{eqnarray}}
\newcommand{\ens}{\end{eqnarray}}

\usepackage{epic, curves}
\def\nn{\nonumber}
\newcommand{\al}{\alpha}

\newcommand{\ep}{\epsilon}

\newtheorem{theorem}{Theorem}[section]

\newtheorem{lemma}{Lemma}[section]

\newtheorem{remark}{Remark}[section]

\def\bq{\begin{equation}}
\def\eq{\end{equation}}
\def\bqq{\begin{eqnarray*}}
	\def\eqq{\end{eqnarray*}}

\def\nn{\nonumber}

\title{\bf Inverse source problem for time fractional diffusion with discrete  random noise   }
\author{ Tuan Nguyen Huy $^1$, Erkan Nane $^2$  \footnote{Corresponding author: \url{ezn0001@auburn.edu }} \\\\
	\small $^{1}$Faculty of Maths and Computer Science, University of Science,\\
	\small Vietnam National University, 227 Nguyen Van Cu, Dist.5, HoChiMinh City, VietNam \\
	\small $^2$ Department of Mathematics and Statistics, Auburn University, Auburn, USA  \\ \\
}
\begin{document}
	\date{}
	\maketitle
	
	\begin{abstract}
		In this paper, we  deal with  the inverse source problem of determining a  source in a time fractional diffusion equation where data are given at a fixed time.
		This problem is ill-posed, i.e., the solution  does not depend continuously on the data.   To regularize the instable solution, we use the trigonometric method in nonparametric
		regression associated with the truncated expansion method. We also investigate the convergence
		rate.
	\end{abstract}


\section{Introduction}

In this work, we focus on an inverse problem for the following time-fractional diffusion equation with a source:

\bes\label{x1}
\left\{\begin{gathered}
\frac{\partial^\al u }{\partial t}-\frac{\partial^2 u }{\partial x^2} = F(x,t),~~(x,t) \in (0,\pi)  \times (0,T) \hfill\\
u_x(0,t) =u_x(\pi,t)=  0,~~  \hfill\\
u(x,0) = 0,~~x \in \Omega  \hfill\\
\end{gathered}  \right.
\ens
where $\Omega=(0,\pi)$, $T > 0$ and $0 < \al < 1$.  Here $\frac{\partial^\al u }{\partial t}$ is the  Caputo fractional derivative of order $\al $ derivative which first appeared in \cite{Caputo} and is defined for an absolutely continuous function $u$ as
\begin{align}
\frac{\partial^\al u(t) }{\partial t} =\frac{1}{\Gamma(1-\al)}\int\limits_{0}^{t}\frac{u^{\prime}(s)}{(t-s)^{\al}}ds,
\end{align}
where  $\Gamma$ denotes the standard Gamma function.  Note that when  the fractional order $\al$ is equal to $1$,
the fractional derivative $\frac{\partial^\al u }{\partial t}$ is equal to the first-order derivative $\frac{du}{dt}$ \cite{Ki}, and thus the problem
\eqref{x1} becomes the classical diffusion problem. When $F(x,t)\equiv 0$, the problem \eqref{x1} was first studied by Nigmatullin \cite{nigmatullin}, and \cite{zaslavsky}. Recently, among many other researchers, Meerchaert et al. \cite{mnv-09} and Baeumer et al. \cite{bmt-2016} have studied the problem \eqref{x1} in a bounded domain in $\mathbb{R}^d$. They also obtained a probabilistic representation of the solution using a time-changed Brownian motion, or other time-changed processes.

Problem \eqref{x1} is a forward problem when the source function $F=F(x,t)$ is given appropriately. Whereas, an inverse source problem based on problem \eqref{x1} is to
determine the source term $F$ at a previous time from its value at a final time $T$ as follows:
\begin{equation}
u(x,T) = u_T(x),~~x \in \Omega. \nonumber
\end{equation}
where  the source function  $F=F(x,t)$ can be split into a product $R( t)f (x)$,  and  $R(t)$ is known in advance.

It is known that the inverse source problem mentioned above is ill-posed in general, i.e., a solution does not always exist. When the solution exists, the solution does not depend continuously on the given initial data. In fact, from a small noise
of a physical measurement, the corresponding solution may have a large error. This
makes the numerical computation  troublesome. Hence a regularization is required.
\\
If $\al = 1,$ the inverse source problem \eqref{x1} is the classical ill-posed problem
and has been studied in \cite{Kirsch, 18}. However, there are only a few works on  the fractional inverse source problem; for example, Sakamoto et al.  \cite{Sa}
used the data $u(x_0, t)(x_0 \in \Omega)$ to determine $R(t)$ once $f (x)$ was given, where the authors obtained a Lipschitz stability for $R(t)$. Wei et al. \cite{Wei2} used the Fourier truncation method to solve an inverse source problem with $R(t) = 1$  in the problem (1.1) for one-dimensional problem with special coefficients. Actually,  there are  very limited number of results for the inverse source  problem for the time-fractional diffusion equation when $R(t)$ depends on time.

{Murio \cite{murio} considered an inverse problem of recovering boundary
	functions from transient data at an interior point in a 1-D semi-infinite half-order time-fractional diffusion equation. Liu and Yamamoto \cite{Liu} applied a
	quasi-reversibility regularization method to solve a backward problem for the time-fractional diffusion equation.   Kirane and  Malik  \cite{Ki1} considered an inverse source problem but they  didn't study  regularization problem.  Recently,    among many others, Jin and  Rundell \cite{Jin},  Wei and  Wang \cite{Wei} and Wnag et al. \cite{Wei1}  have studied  an inverse problem for the time fractional diffusion. Tuan et al \cite{Tuan} considered the  inverse source problem \eqref{x1} for the deterministic case using Tikhonov regularization.}

As is well-known, measurements always are given at a discrete set of points
and contain errors. These errors may be generated from controllable sources or uncontrollable sources. In the first case, the error is often deterministic.  Hence, if we know an approximation  $u_T^\ep$ of the final data $u_T$  then we can construct an approximation $f^\ep$ for  the function $f$. If the errors are generated from uncontrollable sources as wind, rain, humidity, etc, then the model is random. Methods used for the deterministic cases cannot be  applied directly to the random  case.  Because of the random noise, the calculation is often intractable. In practical situations, the function
$u_T(x)$ is a result of experimental measurements  and  it cannot   be observed without errors: hence,  in general we have
\begin{equation}
\widetilde u_T(x_k)=u_T(x_k)+\epsilon_k, \quad  k=1,\cdots, n
\end{equation}
where $\epsilon_k, k=1,\cdots, n$ are unknown independent random errors. In fact, these errors can come from many sources like the measuring instrument or the environment.


\noindent From now on, we put \[x_k = \pi\dfrac{2k-1}{2n},\qquad \text{ with } k = 1,\cdots, n.\]
We have a data set $D = \Big(\widetilde {u_T}(x_1), \widetilde {u_T}(x_2),\ldots, \widetilde {u_T}(x_n) \Big)$, which is the measure of   $$\Big(u_T(x_1), u_T(x_2),\ldots,u_T(x_n)\Big).$$ Here  $D$ satisfies
\begin{equation}\label{x10}
\widetilde {u_T}(x_k) = u_T(x_k) +\sigma_k \epsilon_k,
\end{equation}
where, $\epsilon_k, k = 1,\cdots, n$ are unknown independent noises. Hence  $\epsilon_{k}\sim \mathcal{N}(0,1)$,  and $\sigma_{k} $ are unknown positive constants which are bounded by a positive constant $V_{max}$, i.e., $0 \leq \sigma_{k} < V_{\text{max}}$ for all $k=1,\cdots, n$. The noises $ \epsilon_{k}$ are mutually independent. To the best of our  knowledge,  there does not exist  any  results on inverse source problem for fractional diffusion with random noise in the literature.
Our main problem in this paper is  finding the source function $f$ from the random data $u_T(x_k), k=1,\cdots, n$.

Our main result in this paper is the following theorem.


\begin{theorem}\label{Exp1}
	Let $\epsilon>0$ and $\epsilon_k\sim N(0,1 )$ be independent normal random variables with $ k=1,\cdots, n$ (as mentioned  above).
	Assume that there exists $\beta >0$ and $E>0$ such that
	\begin{equation}
	\|f\|_{H^\beta(\Omega)} \le E,
	\end{equation}
		where
		$\left\|f \right\|_{H^\beta(\Omega)}^2 = \sum_{p=1}^{\infty}  p^{2\beta} \left|\big\langle f,\phi_p\big\rangle\right|^2   .  $
	 Then
	a regularized function $	\widetilde f_{n,M}$ for $f$ can be computed as follows
	\begin{equation}\label{eq4}
	\widetilde f_{n,M}(x) =\frac{{\dfrac{1}{n}\sum_{k=1}^n \widetilde u_T(x_k)  } } {\int\limits_{0}^{T}(T-s)^{\al-1} R(s)ds}+\sum_{p=1}^M  \frac{\dfrac{\pi}{n}\sum_{k=1}^n \widetilde u_T(x_k)\phi_p(x_k) }{\int\limits_{0}^{T}(T-s)^{\al-1}E_{\al,\al}(-p^2 (T-s)^{\al}) R(s)ds} \phi_p(x)\nn\\
	\end{equation}
	where $\phi_p(x)=\sqrt{\frac{2}{\pi}} \cos(px)$ for $p=0, 1,2,3,\cdots$ is a sequence of an orthonormal basis of $L^2(\Omega)$  and  the natural numbers $n$, $M$ are called regularization parameters. Moreover, we have the following estimate
	
	\begin{align}
	\mathbb{E}	\left\|	\widetilde {f}_{n,M}(x) - f(x)\right\|^2_{L^2(\Omega)} &\le \frac{(2-\al)^2}{ R_0^2 T^{4-2\al}} \left(\dfrac{\pi^2 V^2_{\max}}{n} + \frac{\pi^3}{288} \frac{{{\left\| R \right\|}_{\infty}^2}E^2}{n^4}\right)\nn\\
	&+\frac{1} { R_0^2 [1 - E_{\al,1}(-T^{\al}) ]^2 } \Big(\dfrac{\pi^2 V^2_{\max}}{n}+    \frac{\pi^4{\|R\|_{\infty}^2 E^2}}{144 n^4}\Big) M^5\nn\\
	&+M^{-2\beta} E^2.
	\end{align}

	Let  $M:=M_n$ such that $0<M:=M_n <n$ and
	\begin{equation}
	\lim_{n  \to +\infty} \frac{M^5}{n}=0
	\end{equation}
then
\begin{equation}
 \mathbb{E}	\left\|	\widetilde {f}_{n,M}(x) - f(x)\right\|^2_{L^2(\Omega)} \text{is of order}\quad \max\Big(\frac{M^5}{n} , M^{-2\beta}\Big).  \label{mainerror}
\end{equation}
\end{theorem}

\begin{remark}
	By choosing $M := n^{\frac{1}{5+2\beta}}$, and by \eqref{mainerror}, we can conclude that
	\begin{equation}
	\mathbb{E}	\left\|	\widetilde {f}_{n,M}(x) - f(x)\right\|^2_{L^2(\Omega)} \text{is of order} \quad  \Big(\frac{1}{n}\Big)^{\frac{2\beta}{5+2 \beta}}.
	\end{equation}
\end{remark}

\begin{remark}
 The methods of proof in this paper can be used to handle the equation with Laplacian in higher dimensions and L\'evy operators  corresponding to a wide range of L\'evy processes. In particular, this corresponds to replacing the second derivative in equation \eqref{x1} with Laplacian,  or a fractional Laplacian in a bounded domain in $\mathbb{R}^d$ with Dirichlet or Neuman boundary conditions.
\end{remark}
We give the proof of the main result in the next section. We prove a sequence of Lemmas that are of interest in their own right as well.
\section{Proof of main results}
The proof of the main result follows from a couple of Lemmas.
\begin{lemma}\label{derivative-ML-function}(see \cite{Po})
	Let $\lambda > 0$, then we have:
	\bes
	\frac{d}{dt} E_{\al,1}(-\lambda t^{\al})= - \lambda t^{\al-1} E_{\al,\al}(-\lambda t^{\al}),~~ t > 0, 0 < \al < 1. \nn
	\ens
\end{lemma}

\begin{lemma}\label{mittag-leffler-bound}
	Let $R: [0,T] \to \mathbb{R}$ be a positive continuous function such that $\inf_{t \in [0,T]} |R(t)| =R_0 >0 $. Set $ \|R\|_{\infty}=\sup_{t \in [0,T]} |R(t)|   $. Then we have for all $p \in \mathbb{N}$
	\begin{align} \nonumber
	\frac{ R_0 [1 - E_{\al,1}(-T^{\al}) ] }{  p^2 } &\le \int\limits_{0}^{T} (T-s)^{\al-1}E_{\al,\al}(-p^2 (T-s)^{\al}) R(s)ds \le  \frac{\|R\|_{\infty}}{  p^2}.
	\end{align}
\end{lemma}

\begin{proof}
	Using Lemma 2.1,  we have
	\begin{align}
	\int\limits_{0}^{T} (T-s)^{\al-1}E_{\al,\al}(-p^2(T-s)^{\al})ds&=\int\limits_{0}^{T} \Big| s^{\al-1} E_{\al,\al}(-p^2 s^{\al} ) \Big| ds
	= -\frac{1}{p^2} \int\limits_{0}^{T} \frac{d}{ds} E_{\al,1}(-p^2s^{\al}) ds \nn\\
	&= \frac{ 1 - E_{\al,1}(-p^2 T^{\al})}{p^2} \le  \frac{ 1 } {p^2} . \nn
	\end{align}
	This implies that
	\begin{align}
	\int\limits_{0}^{T} (T-s)^{\al-1}E_{\al,\al}(-p^2(T-s)^{\al})R(s) ds &\le  \sup_{t \in [0,T]} |R(t)| \int\limits_{0}^{T} (T-s)^{\al-1}E_{\al,\al}(-p^2(T-s)^{\al}) ds\nn\\
	& \le \frac{  \|R\|_{\infty} } {p^2}.
	\end{align}
	and
	\begin{align}
	\int\limits_{0}^{T} (T-s)^{\al-1}E_{\al,\al}(-p^2(T-s)^{\al}) R(s)ds &\ge \Big( \inf_{t \in [0,T]}  |R(t)| \Big) \int\limits_{0}^{T} (T-s)^{\al-1}E_{\al,\al}(-p^2(T-s)^{\al}) ds \nn\\
	&\ge R_0  \frac{   1 - E_{\al,1}(-p^2T^{\al})   } {p^2} \ge R_0  \frac{   1 - E_{\al,1}(-T^{\al})   } {p^2}.
	\end{align}
in the last inequality we use the fact that $E_{\al,1}(-p^2T^{\al}) \le E_{\al,1}(-T^{\al}) $ for $p \ge 1$; see, for example, \cite{30}.
	The proof is completed.
\end{proof}

\begin{lemma}[\cite{Eubank}, page 144]\label{a1}
	Let $p = 1,\ldots,n-1$,  and $q = 1, 2, \ldots $, with $x_k = \pi\dfrac{2k-1}{2n}$ and $\phi_p(x_k) = \sqrt{\dfrac{2}{\pi}}\cos(px_k)$, then  we have
	\begin{align}
	s_{p,q} = \frac{\sum_{k=1}^n{\phi_p(x_k)\phi_q(x_k)}}{n} = \begin{cases}
	\text{ } \text{ } \text{ } \text{ }\dfrac{1}{\pi}, &  q - p =2 l n \text{ or } q + p =2 l n \text{ ($l$  even)},\\[8pt]
	\text{ }-\dfrac{1}{\pi}, & q - p =2 l n \text{ or } q + p =2 l n \text{ ($l$ odd)},\\
	\text{ } \text{ } \text{ }\text{ } \text{ } 0, & \text{otherwise.}
	\end{cases}
	\end{align}
	If $q = 1, 2 , \ldots, n-1$, then
	\begin{equation}\label{x11}
	s_{p,q} = \begin{cases}
	\text{ } \text{ } \text{ } \dfrac{1}{\pi}, & \qquad p=q,\\
	\text{ } \text{ } \text{ } 0, & \qquad p \ne q.
	\end{cases}
	\end{equation}
	and
	\begin{align}
	\dfrac{1}{n}\sum_{k=1}^n \phi_p(x_k) = \begin{cases}
	0, & \qquad p \ne 2ln,\\
	(-1)^l\sqrt{\dfrac{2}{\pi}}, & \qquad p = 2ln. \nonumber
	\end{cases}
	\end{align}
\end{lemma}
 From  this lemma, we have the next  result.

\begin{lemma}\label{cc}
	Let $p, n \in \mathbb{N}$ such that $0\le p \le n-1$. Assume that $u_T$ is piecewise $C^1$ on $[0,\pi]$ .
	Then
	\begin{equation}\label{y1111}
	\Big \langle u_T\left(x\right), \phi_p(x) \Big\rangle= \begin{cases}
	\text{ } \text{ } \text{ } \dfrac{1}{n}\sum_{k=1}^n u_T(x_k) - \widetilde {G}_{n0}, & \qquad p=0,\\\\
	\text{ } \text{ } \text{ }  \dfrac{\pi}{n} \sum_{k=1}^n u_T(x_k) \phi_p(x_k) -  \widetilde {G}_{np}, & \qquad 1 \le p \le n-1.
	\end{cases}
	\end{equation}
	where
	\begin{equation}\label{useful2}
	\widetilde {G}_{np}= \begin{cases}
	\text{ } \text{ } \text{ } \sqrt{\dfrac{2}{\pi}}\sum_{l=1}^\infty (-1)^ l \Big \langle u_T\left(x\right), \phi_{2ln}(x) \Big\rangle, & \qquad p=0,\\\\
	\sum_{l=1}^\infty (-1)^l \Bigg[  \Big \langle u_T\left(x\right), \phi_{p+2ln}(x) \Big\rangle + \Big \langle u_T\left(x\right), \phi_{-p+2ln}(x) \Big\rangle\Bigg], & \qquad 1 \le p \le n-1.
	\end{cases}
	\end{equation}
\end{lemma}

\begin{proof}
Using the complete orthonormal basis $\{\phi_p; p=0,1,2, \cdots\}$, the function $u_T$ can be written as follows  $$u_T(x_k)=\big \langle u_T\left(x\right), \phi_0(x) \Big\rangle+ \sum_{q=1}^\infty \left \langle u_T \left(x\right), \phi_q(x) \right\rangle\phi_q(x_k) .$$
	This implies that
	\begin{align}
	\dfrac{1}{n}\sum_{k=1}^n u_T(x_k)\phi_p(x_k) &= \dfrac{1}{n}\sum_{k=1}^n \big \langle u_T \left(x\right), \phi_0(x) \Big\rangle \phi_p(x_k) +\dfrac{1}{n}\sum_{k=1}^n\left(\sum_{q=1}^\infty \left \langle u_T\left(x\right), \phi_q(x) \right\rangle\phi_q(x_k)\right)\phi_p(x_k). \nonumber
	\end{align}
	For $p \ge 1$,   using {Lemma \ref{a1} }, we get
	\begin{align}
	\dfrac{1}{n}\sum_{k=1}^n u_T  (x_k)\phi_p(x_k) &= \dfrac{1}{n}\sum_{k=1}^n\left(\sum_{q=1}^n \left \langle u_T\left(x\right), \phi_q(x) \right\rangle \phi_q(x_k)\right)\phi_p(x_k) \nn\\
	&+ \dfrac{1}{n}\sum_{k=1}^n\left(\sum_{q=n+1}^\infty \left \langle u_T \left(x\right), \phi_q(x) \right\rangle \phi_q(x_k)\right)\phi_p(x_j)\nn\\
	&=\dfrac{1}{n}\sum_{q=1}^n \left \langle u_T \left(x\right), \phi_q(x) \right\rangle \Big[\sum_{k=1}^n{\phi_q(x_k)\phi_p(x_k)} \Big]\nn\\
	&\quad \quad +\dfrac{1}{n}\sum_{q=n+1}^\infty \left \langle u_T \left(x\right), \phi_q(x) \right\rangle \Big[\sum_{k=1}^n{\phi_q(x_k)\phi_p(x_k)} \Big]. \label{ine1}
	\end{align}
	By {Lemma \ref{a1}}, we obtain for
	$1 \le q \le n$
	\begin{equation}\label{x1111}
	\left \langle u_T\left(x\right), \phi_q(x) \right\rangle \Big[\sum_{k=1}^n{\phi_q(x_k)\phi_p(x_k)} \Big]= \begin{cases}
	\text{ } \text{ } \text{ } \dfrac{\left \langle u_T\left(x\right), \phi_q(x) \right\rangle n }{\pi}, & \qquad p=q,\\
	\text{ } \text{ } \text{ } 0, & \qquad p \ne q.
	\end{cases}
	\end{equation}
	For $q \ge n$
	\begin{equation}\label{x11111}
	\left \langle u_T\left(x\right), \phi_q(x) \right\rangle \Big[\sum_{k=1}^n{\phi_q(x_k)\phi_p(x_k)} \Big]= \begin{cases}
	\text{ } \text{ } \text{ } \frac{(-1)^l n\left \langle u_T\left(x\right), \phi_q(x) \right\rangle}{\pi} , \quad \text{if} \quad q=p+2ln \quad \text{or}\quad  q=-p+2ln,\hfill \\
	0, \quad \quad  \text{otherwise}.
	\end{cases}
	\end{equation}
	Hence
	\begin{align}
	\dfrac{1}{n}\sum_{q=1}^n \left \langle u_T\left(x\right), \phi_q(x) \right\rangle \Big[\sum_{k=1}^n{\phi_q(x_k)\phi_p(x_k)} \Big]= \frac{\left \langle f\left(x\right), \phi_p(x) \right\rangle}{ \pi}. \label{ine2}
	\end{align}
	This implies that
	\begin{align}
	\dfrac{1}{n}\sum_{q=n+1}^\infty \left \langle u_T\left(x\right), \phi_q(x) \right\rangle \Big[\sum_{k=1}^n{\phi_q(x_k)\phi_p(x_k)} \Big]= \dfrac{1}{\pi} \sum_{l=1}^\infty (-1)^l \Bigg[  \Big \langle u_T\left(x\right), \phi_{p+2ln}(x) \Big\rangle + \Big \langle u_T\left(x\right), \phi_{-p+2ln}(x) \Big\rangle\Bigg]. \label{ine3}
	\end{align}
	Combining \eqref{x11111}, \eqref{ine2}, \eqref{ine3}   we get
	\begin{align}
	\dfrac{1}{n}\sum_{k=1}^n u_T(x_k)\phi_p(x_k) = \dfrac{1}{\pi}\Big[ \left \langle u_T \left(x\right), \phi_q(x) \right\rangle +  \widetilde {G}_{np} \Big]. \label{ine1}
	\end{align}
	Therefore, \eqref{useful2} holds for $1 \le p \le n-1$. Similarly, we have
	\begin{align}
	\sum_{k=1}^nu_T(x_k) &= n \Big \langle u_T\left(x\right), \phi_0(x) \Big \rangle + \sum_{p=1}^\infty \big \langle u_T\left(x\right), \phi_p(x) \Big \rangle\sum_{k=1}^n\phi_p(x_k) \nn\\
	&= n \Big \langle u_T \left(x\right), \phi_0(x) \Big \rangle + n \widetilde {G}_{n0}. \nonumber
	\end{align}
	This completes the proof of Lemma~\ref{cc}.
\end{proof}

The following Lemma gives the formula of $f$  in terms of  $u_T(x_k)$.

\begin{lemma}\label{thr1}
	Let $0<M<n$, $M \in \mathbb{N}$. Assume that $u_T$ is as in Lemma~\ref{cc}. Then the source function $f$ is given by
	\begin{align}\label{theta1}
	f(x) &=\frac{{\dfrac{1}{n}\sum_{k=1}^n u_T(x_k) - \widetilde {G}_{n0}}} {\int\limits_{0}^{T}(T-s)^{\al-1} R(s)ds}+ \sum_{p=1}^M  \frac{\dfrac{\pi}{n}\sum_{k=1}^n u_T(x_k)\phi_p(x_k)  - \widetilde {G}_{np}}{\int\limits_{0}^{T}(T-s)^{\al-1}E_{\al,\al}(-p^2 (T-s)^{\al}) R(s)ds} \phi_p(x) \nonumber \\
	& +\sum_{p=M+1}^\infty
	\frac{\Big<u_T(x),\phi_{p}(x)\Big>}{\int\limits_{0}^{T}(T-s)^{\al-1}E_{\al,\al}(-p^2 (T-s)^{\al}) R(s)ds} \phi_p(x) .
	\end{align}
\end{lemma}
\begin{proof}
First, we  have the following equality
	\begin{align}
	f(x) =  \frac{\Big \langle u_T\left(x\right), \phi_0(x) \Big\rangle}{ \int_0^T {{{(T - s)}^{\alpha  - 1}}} R(s)ds   } + \sum\limits_{p=1}^{\infty} \frac{\Big<u_T(x),\phi_{p}(x)\Big>\phi_{p}(x)  } { \int\limits_{0}^{T}(T-s)^{\al-1}E_{\al,\al}(-p^2 (T-s)^{\al}) R(s)ds} . \label{f-bound-eq1}
	\end{align}

	{To prove \eqref{f-bound-eq1}, we   use the results of  Sakamoto and Yamamoto \cite{Sa}. According Theorem 2.4 in \cite{Sa},  the solution of  $\eqref{x1}$ satisfies
	\begin{equation}
	\Big<u(x,t),\phi_{p}(x)\Big> = E_{\al,1}(-p^2 t^{\al}) 	\Big<u(x,0),\phi_{p}(x)\Big> + \int\limits_{0}^{t}(t-s)^{\al-1}E_{\al,\al}(-p^2 (t-s)^{\al})\Big<F(x,s),\phi_{p}(x)\Big> ds.
	\end{equation}
By letting $t=T$ in the last equality,  recalling  $u(x,0) = 0$ and $F_p(s)= R(s) \Big<f(x),\phi_{p}(x)\Big>$ , we get
	\begin{equation}
	\Big<u(x,T),\phi_{p}(x)\Big> =\Big<f(x),\phi_{p}(x)\Big> \int\limits_{0}^{T}(T-s)^{\al-1}E_{\al,\al}(-p^2 (T-s)^{\al}) R(s)ds .
	\end{equation}
This implies \eqref{f-bound-eq1}. }\\
	Using {Lemma \ref{cc}}, the first term of the right hand side of the last equality is equal to
	\begin{align}
	\Big \langle f\left(x\right), \phi_0(x) \Big\rangle &= \frac{{{\Big \langle u_T\left(x\right), \phi_0(x) \Big\rangle}}}{{\int_0^T {{{(T - s)}^{\alpha  - 1}}} R(s)ds}} = \frac{{\frac{1}{n}\sum\limits_{k = 1}^n u_T ({x_k}) - {\widetilde {G}_{n0} } }}{{\int_0^T {{{(T - s)}^{\alpha  - 1}}} R(s)ds}}
	\end{align}
	and the second term on the right hand side of \eqref{f-bound-eq1} is equal to
	\begin{align}
	&\sum\limits_{p=1}^{\infty}\frac{\Big<u_T\left(x\right),\phi_{p}(x)\Big>\phi_{p}(x)  }{\int\limits_{0}^{T}(T-s)^{\al-1}E_{\al,\al}(-p^2 (T-s)^{\al}) R(s)ds} \nn\\
	&= \sum\limits_{p=1}^{M}\frac{\Big<u_T\left(x\right),\phi_{p}(x)\Big>\phi_{p}(x)  }{\int\limits_{0}^{T}(T-s)^{\al-1}E_{\al,\al}(-p^2 (T-s)^{\al}) R(s)ds} + \sum\limits_{p=M+1}^{\infty}\frac{\Big<u_T\left(x\right),\phi_{p}(x)\Big>\phi_{p}(x)  }{\int\limits_{0}^{T}(T-s)^{\al-1}E_{\al,\al}(-p^2 (T-s)^{\al}) R(s)ds}  \nn\\
	&=  \sum\limits_{p=1}^{M}\frac{\left( {\frac{\pi }{n}\sum\limits_{k = 1}^n u_T ({x_k}){\phi _p}(x) - {\widetilde G_{np}}} \right){\phi _p}(x)}{\int\limits_{0}^{T}(T-s)^{\al-1}E_{\al,\al}(-p^2 (T-s)^{\al}) R(s)ds  }  + \sum\limits_{p=M+1}^{\infty}\frac{\Big<u_T(x),\phi_{p}(x)\Big>\phi_{p}(x)  }{\int\limits_{0}^{T}(T-s)^{\al-1}E_{\al,\al}(-p^2 (T-s)^{\al}) R(s)ds  }. \label{lim-eq2}
	\end{align}
	Substituting \eqref{lim-eq2} into~\eqref{f-bound-eq1}, we get~\eqref{theta1}.
\end{proof}

Now, we return the proof of main result.
\begin{proof}[Proof of Theorem \ref{Exp1}]
First, 	we have the following estimate
	\begin{align}
	\Big| {\left\langle {u_T\left( x \right),{\phi _p}(x)} \right\rangle } \Big|& =   \Bigg[ \int\limits_{0}^{T}(T-s)^{\al-1}E_{\al,\al}(-p^2 (T-s)^{\al}) R(s)ds \Bigg]  \Big| {\left\langle {f\left( x \right),{\phi _p}(x)} \right\rangle } \Big|\nn\\
	& \le \frac{{{{\left\| R \right\| }_{\infty}}} \|f\|_{L^2(\Omega)}   }{{{p^2}}} .\label{useful3}
	\end{align}
	Using  \eqref{eq4}  and \eqref{f-bound-eq1}, we obtain
	\begin{align}
	\widetilde {f}_{n,M}(x) - f(x)&=  \frac{{\dfrac{1}{n}\sum_{k=1}^n \sigma_k \ep_k  } +\widetilde G_{n0}} {\int\limits_{0}^{T}(T-s)^{\al-1} R(s)ds}+ \sum\limits_{p=1}^{M} \Bigg[\frac{{\frac{\pi }{n}\sum\limits_{k = 1}^n \sigma_k {{\epsilon _k}{\phi _p}({x_k})}  + {\widetilde {G}_{np}}} }{\int\limits_{0}^{T}(T-s)^{\al-1}E_{\al,\al}(-p^2 (T-s)^{\al}) R(s)ds }\Bigg] \phi_p(x)  \nn\\
	&-\sum_{p=M+1}^\infty
	\Bigg[\frac{\Big<u_T(x),\phi_{p}(x)\Big>}{\int\limits_{0}^{T}(T-s)^{\al-1}E_{\al,\al}(-p^2 (T-s)^{\al}) R(s)ds}\Bigg] \phi_p(x)
	\end{align}
	Applying  Lemma ~\ref{thr1} we obtain
	\begin{align}
	\|	\widetilde {f}_{n,M}(x) - f(x)\|^2_{L^2(\Omega)}&=  \Bigg[ \frac{{\dfrac{1}{n}\sum_{k=1}^n \sigma_k\ep_k  } +\widetilde G_{n0}} {\int\limits_{0}^{T}(T-s)^{\al-1} R(s)ds} \Bigg]^2 + \sum\limits_{p=1}^{M} \Bigg[\frac{{\frac{\pi }{n}\sum\limits_{k = 1}^n \sigma_k {{\varepsilon _k}{\phi _p}({x_k})}  + {\widetilde {G}_{np}}} }{\int\limits_{0}^{T}(T-s)^{\al-1}E_{\al,\al}(-p^2 (T-s)^{\al}) R(s)ds } \Bigg]^2\nonumber \\
	&+\sum_{p=M+1}^\infty\Bigg[  \frac{\Big<u_T(x),\phi_{p}(x)\Big>}{\int\limits_{0}^{T}(T-s)^{\al-1}E_{\al,\al}(-p^2 (T-s)^{\al}) R(s)ds}\Bigg]^2  \nonumber
	\end{align}
	This follows from  the Parseval identity and the  fact that $\mathbb{E}(\epsilon_j \epsilon_l) = 0$; $(j \ne l)$, and $\mathbb{E}\epsilon_j = 0; j = 1,2, \ldots ,n$.  Now
	
	\begin{align}  \label{upper-eq1}
	&\mathbb{E}	\left\|	\widetilde {f}_{n,M}(x) - f(x)\right\|^2_{L^2(\Omega)}\nn\\
	 &=\underbrace {\frac{\dfrac{1}{n^2}\sum_{k=1}^n \sigma_k^2 \mathbb{E}\epsilon_k^2 + \widetilde G_{n0}^2} {  \left( \int_{0}^{T}(T-s)^{\alpha-1}R(s)ds \right)^2 } }_{=: I_1}
	+\underbrace{ \sum_{p=M+1}^\infty \left[ {\dfrac{\langle u_T(x),\phi_p(x) \rangle}{\int\limits_{0}^{T}(T-s)^{\al-1}E_{\al,\al}(-p^2 (T-s)^{\al}) R(s)ds}} \right]^2 }_{=:I_2} \nonumber \\
	& +\underbrace{ \sum_{p=1}^M  \frac{\dfrac{\pi^2}{n^2}\sum_{k=1}^n\sigma_k^2 \mathbb{E}\epsilon_k^2 + \widetilde G_{np}^2 }{\bigg[\int\limits_{0}^{T}(T-s)^{\al-1}E_{\al,\al}(-p^2 (T-s)^{\al}) R(s)ds\bigg]^2} }_{=:I_3}.
	\end{align}
	First, by  \eqref{useful2} and \eqref{useful3} we know that
	\begin{equation}
	\widetilde G_{n0} \le \sqrt{\dfrac{2}{\pi}}\sum_{l=1}^\infty \left| {\langle u_T\left( x \right),{\phi _{2ln}}(x)\rangle } \right|
	\le \sqrt {\frac{2}{\pi }} \sum\limits_{l = 1}^\infty  \frac{ \|R\|_{\infty}  \|f\|_{L^2(\Omega)}  }{4l^2 n^2} =\sqrt {\frac{2}{\pi }} \frac{\pi^2}{24} \frac{ \|R\|_{\infty} \|f\|_{L^2(\Omega)}  }{n^2} .
	\end{equation}
	where we use the fact that $\sum_{l=1}^{\infty} \frac{1}{l^2}= \frac{\pi^2}{6}$.
	By a similar method  as above, we  obtain
	\begin{align}
	\widetilde G_{np} &\le \sum_{l=1}^\infty  \Bigg|  \Big \langle u_T\left(x\right), \phi_{p+2ln}(x) \Big\rangle + \Big \langle u_T\left(x\right), \phi_{-p+2ln}(x) \Big\rangle\Bigg|\nn\\
	&\le   {\|R\|_{\infty} \|f\|_{L^2(\Omega)} }  \Bigg[\sum_{l=1}^\infty \frac{1}{(p+2ln)^2} +\sum_{l=1}^\infty \frac{1}{(-p+2ln)^2}\Bigg]\nn\\
	& \le \frac{\pi^2}{12}    \frac{{\|R\|_{\infty} \|f\|_{L^2(\Omega)} }}{n^2}.
	\end{align}
	Since  $\sigma_k < V_{\max}$, we estimate $I_1$ as follows
	\begin{align}
	I_1 \le \frac{ \dfrac{ V^2_{\max}}{n} + \widetilde G_{n0}^2 }{ \pi^2 \left( \int_{0}^{T}(T-s)^{\alpha-1}R(s)ds \right)^2  } \le \frac{(2-\al)^2}{ R_0 ^2 T^{4-2\al}} \left(\dfrac{ V^2_{\max}}{n} + \frac{\pi^3}{288} \frac{{{\left\| R \right\|}_{\infty}^2}\left\| f \right\|^2_{L^2(\Omega)} }{n^4}\right) .\label{max-eq2}
	\end{align}

 {	By equation  \eqref{f-bound-eq1}, we know that for $p \ge 1$
	\begin{equation}
	<f(x), \phi_p(x)>=	 {\dfrac{\langle u_T(x),\phi_p(x) \rangle}{\int\limits_{0}^{T}(T-s)^{\al-1}E_{\al,\al}(-p^2 (T-s)^{\al}) R(s)ds}}.
	\end{equation}
}	


 {
Recall  the definition of $I_2$ in \eqref{upper-eq1}:
\begin{align}
I_2=\sum_{p=M+1}^\infty \left[ {\dfrac{\langle u_T(x),\phi_p(x) \rangle}{\int\limits_{0}^{T}(T-s)^{\al-1}E_{\al,\al}(-p^2 (T-s)^{\al}) R(s)ds}} \right]^2 .
\end{align}
}

	 {
	using the last two equations,  we get
		\begin{align}
		I_2=\sum_{p=M+1}^\infty \Big[<f(x), \phi_p(x)> \Big]^2 .
		\end{align}
	}
	
	
		 {
		Since  $1=p^{-2\beta} p^{2\beta} $, we can rewrite $I_2$ as follows
			\begin{equation}\label{i2-sum}
			I_2=\sum_{p=M+1}^\infty p^{-2\beta} p^{2\beta}\Big|<f(x), \phi_p(x)> \Big|^2.
			\end{equation}
		}


		 {
In the last series \eqref{i2-sum}  since $p \ge M+1 >M$, we get $p^{-2\beta} \le M^{-2\beta}$.  	
}


 {
Using the last two  observations, we obtain
\begin{equation}\label{i2-tilde-1}
I_2 \le \sum_{p=M+1}^\infty  M^{-2\beta} p^{2\beta}\Big|<f(x), \phi_p(x)> \Big|^2 \phi_p^2(x)= M^{-2\beta} \underbrace{\sum_{p=M+1}^\infty p^{2\beta}\Big|<f(x), \phi_p(x)> \Big|^2  }_{:=\widetilde I_2}.
\end{equation}
}

	
 {	
	It is easy to see that
	\bes\label{i2-tilde-2}
	\widetilde I_2 = \sum_{p=M+1}^\infty p^{2\beta}\Big|<f(x), \phi_p(x)> \Big|^2  \le \sum_{p=1}^\infty p^{2\beta}\Big|<f(x), \phi_p(x)> \Big|^2 = \left\|f \right\|_{H^\beta(\Omega)}^2.
	\ens
}	

 {	
	Using \eqref{i2-tilde-1} and \eqref{i2-tilde-2}, we get
	\begin{align}
	I_2 &\le   M^{-2\beta} \left\|f \right\|_{H^\beta(\Omega)}^2.
	\label{I2-eq3}
	\end{align}
	 }

	Using Lemma \ref{mittag-leffler-bound},  the $I_3$ term  can be estimated as follows
	\begin{align}
	I_3 &\le \Bigg(\dfrac{\pi^2 V^2_{\max}}{n}+  \frac{\pi^4}{144}    \frac{{\|R\|_{\infty}^2 \|f\|^2_{L^2(\Omega)}  }}{n^4}\Bigg) \sum_{p=1}^M  {\bigg[\int\limits_{0}^{T}(T-s)^{\al-1}E_{\al,\al}(-p^2 (T-s)^{\al}) R(s)ds\bigg]^{-2}}  \nn\\
	&\le\Bigg(\dfrac{\pi^2 V^2_{\max}}{n}+  \frac{\pi^4}{144}    \frac{{\|R\|_{\infty}^2 \|f\|^2_{L^2(\Omega)}  }}{n^4}\Bigg) \sum_{p=1}^M \frac{  p^4 }{ R_0^2 [1 - E_{\al,1}(-T^{\al}) ]^2 } \nn\\
	&\le\frac{1} { R_0^2 [1 - E_{\al,1}(-T^{\al}) ]^2 } \Big(\dfrac{\pi^2 V^2_{\max}}{n}+    \frac{\pi^4{\|R\|_{\infty}^2 \|f\|^2_{L^2(\Omega)}  }}{144 n^4}\Big) M^5. \label{I3-eq4}
	\end{align}
	
	Combining \eqref{upper-eq1}, \eqref{max-eq2}, \eqref{I2-eq3}, \eqref{I3-eq4}, we obtain
	\begin{align}
	\mathbb{E}	\left\|	\widetilde {f}_{n,M}(x) - f(x)\right\|^2_{L^2(\Omega)} &\le  \frac{(2-\al)^2}{ R_0^2 T^{4-2\al}} \left(\dfrac{\pi^2 V^2_{\max}}{n} + \frac{\pi^3}{288} \frac{{{\left\| R \right\|}_{\infty}^2}\left\| f \right\|^2_{L^2(\Omega)}  }{n^4}\right)\nn\\
	&+\frac{1} { R_0^2 [1 - E_{\al,1}(-T^{\al}) ]^2 } \Big(\dfrac{\pi^2 V^2_{\max}}{n}+    \frac{\pi^4{\|R\|_{\infty}^2 \|f\|^2_{L^2(\Omega)}   }}{144 n^4}\Big) M^5\nn\\
	&+M^{-2\beta} \left\|f \right\|_{H^\beta}^2.
	\end{align}
	This completes the proof.
\end{proof}

\section*{Acknowledgment}
The first author  gratefully acknowledge stimulating discussions with
 Prof Dang Duc Trong.

\end{document}